\documentclass[12pt]{amsart}

\usepackage{graphicx}
\usepackage{amsmath}
\usepackage{amssymb}

\usepackage{ragged2e}
\newtheorem{theorem}{Theorem}[section]

\theoremstyle{plain}
\newtheorem{lemma}[theorem]{Lemma}
\newtheorem{corollary}{Corollary}[theorem]
\newtheorem{prop}[theorem]{Proposition}

\theoremstyle{definition}
\newtheorem{definition}{Definition}[section]

\newtheorem{exmp}{Example}[section]

\theoremstyle{remark}
\newtheorem*{remark}{Remark}
\begin{document}

\title[$n-$absorbing $I-$prime]{$n-$absorbing $I-$prime hyperideals in multiplicative hyperrings}

\author[A. A. Mena]{Ali A. Mena}
\address{Mathematics Department, Faculty of Science\\
	Soran University\\ 44008, Soran, Erbil
	Kurdistan Region, Iraq}
\email {ali07863@gmail.com}

\author[I. Akray]{Ismael Akray}
\address{Mathematics Department, Faculty of Science\\
	Soran University\\ 44008, Soran, Erbil
	Kurdistan Region, Iraq}

\email {ismael.akray@soran.edu.iq}

\date{\today}
\maketitle

\begin{abstract}
In this paper, we define the concept   $I-$prime hyperideal in a multiplicative hyperring $R$.  A proper hyperideal $P$ of $R$ is  an $I-$prime hyperideal if for $a, b \in R$ with $ab \subseteq P-IP$ implies  $a \in P$ or $b \in P$. We provide some characterizations of $I-$prime hyperideals. Also we conceptualize and study the notions  $2-$absorbing $I-$prime and $n-$absorbing $I-$prime hyperideals into multiplicative hyperrings as generalizations of prime ideals. A proper hyperideal $P$ of a hyperring $R$ is an $n-$absorbing $I-$prime hyperideal if for $x_1, \cdots,x_{n+1} \in R$ such that $x_1 \cdots x_{n+1} \subseteq P-IP$, then $x_1 \cdots x_{i-1} x_{i+1} \cdots   x_{n+1} \subseteq P$ for some $i \in \{1, \cdots ,n+1\}$. We  study some properties of such generalizations.  We prove that if $P$ is an $I-$prime hyperideal of a hyperring $R$, then each of $\frac{P}{J}$, $S^{-1} P$, $f(P)$, $f^{-1}(P)$, $\sqrt{P}$ and $P[x]$ are $I-$prime hyperideals under suitable conditions and suitable hyperideal $I$, where $J$ is a hyperideal contains in $P$. Also, we characterize $I-$prime hyperideals in the decomposite hyperrings. Moreover, we show that the hyperring with finite number of maximal hyperideals in which every proper hyperideal is $n-$absorbing $I-$prime is a finite product of hyperfields.
\end{abstract}

\section{Introduction}
 Many concepts in modern algebra was generalized by generalizing their structures to hyperstructure. The French mathematician F. Marty in 1934 introduced the concept hyperstructure or multioperation by returning a set of values instead of a single value \cite{marty}. The hyperstructures theory was studied from many points of view and applied to several areas of mathematics especially in computer science and logic. In \cite{marty} the author presented the concept hypergroup and after that in 1937, the authors H. S. Wall \cite{wall} and M. Kranser  \cite{Krasner} also gave their respective definitions of hypergroup as a generalization of groups.
 
 \vspace{0.5 cm}
 The hyperrings were introduced by many authors. A type of hyperring where the multiplication is a hyperoperation while the addition is just an operation introduced by Rota in 1982 and called a multiplicative hyperring \cite{rota}. A well known example on multiplicative hyperring is that for a ring $(R,+, \cdot)$ and  corresponding to every non-singleton subset $A \in P^{*}(R)=P(R) \backslash\{\phi\}$ where $P(R)$ is the power set of $R$, there exists a multiplicative hyperring with absorbing zero $\left(R_{A},+, \circ\right)$ where $R_{A}=R$ and for any $x, y \in R_{A}, x \circ y=\{x \cdot a \cdot y: a \in A\} $ (see \cite{procesi,vougiouklis}). Another type of hyperring in which addition is a hyperoperation while the multiplication is an operation introduced by  M. Krasner in 1983 and called Krasner hyperring \cite{Krasner}. The hyperrings in which the additions and multiplications are hyperoperations where introduced by De Salvo \cite{Desalvo}. Procesi and Rota in \cite{procesi} have conceptualized the notion of primeness of hyperideal in a multiplicative hyperring. A proper hyperideal $P$ is called prime hyperideal if $ab\subseteq P$, then $a \in P$ or $b \in P$. The radical of a hyperideal $P$ denoted by $\sqrt{P}$ is the intersection of all prime hyperideals that contains $P$. Some generalizations of prime hyperideals can be found in \cite{Anbarloei,Dasgupta,Ulucak} .
 \vspace{0.5 cm}
 
 In the recent years many generalizations of prime ideals were introduced. Here state some of them. The authors in \cite{badwi} and \cite{anderson} introduced the notions $2-$absorbing and $n-$absorbing ideals in commutative rings. A proper ideal $P$ is called $2-$absorbing (or $n-$absorbing) ideal if whenever the product of three (or $n+1$) elements of $R$ in $P$, the product of two (or $n$) of these elements is in $P$.  
 
 In \cite{Akray} and \cite{akraym}, the author Akray introduced the notions $I-$prime ideal and $n-$absorbing $I-$ideal in classical rings as a generalization of prime ideals. For fixed proper ideal $I$ of a commutative ring $R$ with identity, a proper ideal $P$ of $R$ is an $I-$prime if for $a, b \in R$  with $a.b \in P-IP$, then $a \in P$ or $b \in P$. A proper ideal $P$ of $R$ is an $n-$absorbing $I-$ideal if for $x_1, \cdots,x_{n+1} \in R$ such that $x_1 \cdots x_{n+1} \in P-IP$, then $x_1 \cdots x_{i-1} x_{i+1} \cdots  x_{n+1} \in P$ for some $i \in \{1,2, \cdots ,n+1\}$.
 \vspace{0.5 cm} 
 
 In this paper all hyperrings are commutative hyperring with identity. Here we want to define the $I-$prime hyperideal, $2-$absorbing $I-$hyperideal and $n-$absorbing $I-$hyperideal in multiplicative hyperrings. For fixed proper hyperideal $I$ of a multiplicative hyperring $R$, a proper hyperideal $P$ of $R$ is an $I-$prime  if $a, b \in R$ with $a.b \subseteq P-IP$, then $a \in P$ or $b \in P$. A proper hyperideal $P$ of $R$ is a $2-$absorbing $I-$prime hyperideal if for $x_1,x_2,x_3 \in R$ such that $x_1 x_2 x_3 \subseteq P-IP$, then $x_1 x_2 \subseteq P$ or $x_1 x_3 \subseteq P$  or $x_2 x_3 \subseteq P$. A proper hyperideal $P$ of $R$ is an $n-$absorbing $I-$prime hyperideal if for $x_1, \cdots,x_{n+1} \in R$ such that $x_1 \cdots x_{n+1} \subseteq P-IP$, then $x_1 \cdots x_{i-1} x_{i+1} \cdots  x_{n+1} \subseteq P$ for some $i \in \{1,2, \cdots ,n+1\}$.
  \vspace{0.5 cm}
   
 In section two, we define $I-$prime hyperideal and we prove some equivalents of $I-$prime hyperideal (Theorem \ref{qqq}). Moreover, we establish $I-$prime hyperideals in finite product of hyperrings (Theorem \ref{qqqqq}). 
 Section three devoted for $2-$absorbing $I-$prime and $n-$absorbing $I-$prime hyperideals and we prove Theorem \ref{3.7} 
 which state (Let $R=\prod_{i=1}^{n+1}R_i$ and $P$ be a proper non-zero hyperideal of $R$. If P is an $(n+1)-$absorbing $I-$prime hyperideal of $R$, then $P=P_1\times P_2\times \cdots \times P_{n+1}$ for some proper $n-$absobing $I_i-$prime hyperideals $P_1,\cdots ,P_{n+1}$ of $R_1,\cdots,R_{n+1}$ respectively, where  $I=\prod_{i=1}^{n+1}I_i$ and $I_i = R_i$, $ \forall i=1,2,\cdots,n+1$). Also,  we prove Theorem \ref{3.9} that state (Let $\mid Max(R)\mid \geq n+1 \geq 2.$ Then each proper hyperideal of $R$ is an $n-$absorbing $I-$prime hyperideal if and only if each quotient of $R$ is a product of $(n+1)-$fields). Finally, let $P$ be an $n-$absorbing $I-$hyperideal of a hyperring $R$. Then there are at most $n^{th}$ prime hyperideals    of $R$ that are minimal over $P$ (Theorem \ref{3.10}).

\section{$I-$prime hyperideals}

We start this section by defining the concept of $I-$prime hyperideal and some example of it. A proper hyperideal $P$ of $R$ is  an $I-$prime hyperideal if for $a, b \in R$ with $ab \subseteq P-IP$ implies  $a \in P$ or $b \in P$.

 In the following examples we show that the class of $I-$prime hyperideals contains properly the class of prime hyperideals. 
	
	\begin{exmp}
		Consider the hyperring of integers $(\mathbb{Z},+,\circ)$, $ A=\{0,1\} \subseteq \mathbb{Z}$ and $n \circ m= \{ nam: a \in A\}= \{0,nm\}$. So $4\mathbb{Z}$ is not prime hyperideal, since $2\circ 2=\{0,4\} \subseteq 4\mathbb{Z}$ and $2\notin 4\mathbb{Z}$. But $4\mathbb{Z}$ is $2\mathbb{Z}$-prime hyperideal, since $\forall a,b\in \mathbb{Z},a\circ b=\{0,ab\}\nsubseteq$ $4\mathbb{Z}-(2\mathbb{Z}\circ 4\mathbb{Z})= 4\mathbb{Z}-8\mathbb{Z}.$
			\end{exmp}
	\begin{exmp}
		Let $(\mathbb{Z},+,\circ)$ be the hyperring of integers and $A=\{4,8\} \subseteq \mathbb{Z}$ and $a\circ b = aAb=\{4ab,8ab\}$. Then
		$1\circ 1= \{4,8\}\subseteq 2\mathbb{Z}$ but $1\notin 2\mathbb{Z}$
	and hence $2\mathbb{Z}$ is not prime hyperideal. However $2\mathbb{Z}$ is not $8\mathbb{Z}-$prime hyperideal, since 
		$2\mathbb{Z}-(8\mathbb{Z} \circ 2\mathbb{Z})=2\mathbb{Z}-(64\mathbb{Z} \cup 128\mathbb{Z})$ $=2\mathbb{Z}-64\mathbb{Z} $ which contains 	$1\circ 1$. 
Therefore, $2 \mathbb{Z}$ is neither prime hyperideal nor $8\mathbb{Z}$-prime hyperideal of $(\mathbb{Z},+,\circ)$.\\
	\end{exmp}

 The intersection of two $I-$prime hyperideals is not $I-$prime hyperideal let us explain our claim by this example.
\begin{exmp}
		Consider the hyperring of integers $(\mathbb{Z},+,\circ)$,  where $a\circ b$ = $\{2ab,3ab\}$. Let 
	$P=2\mathbb{Z}$, $I=3\mathbb{Z}$ and
	$P-IP=2\mathbb{Z}-(3\mathbb{Z})\circ (2\mathbb{Z})=2\mathbb{Z}-6A \mathbb{Z}=2\mathbb{Z}-(12\mathbb{Z}\cup 18\mathbb{Z})$. 
Thus	$P$ is $I-$prime hyperideal. Now, for $Q=3\mathbb{Z}$ and  $I=3\mathbb{Z}$ we have 
		$Q-IQ=3\mathbb{Z}-(3\mathbb{Z})\circ (3\mathbb{Z})=3\mathbb{Z}-9A \mathbb{Z}=3\mathbb{Z}-(18\mathbb{Z}\cup 27\mathbb{Z})$. So $Q$ is $I-$prime hyperideal of $\mathbb{Z}$ while 		
		$P\cap Q=6\mathbb{Z}$ is not $3\mathbb{Z}-$prime hyperideal, since 	$6\mathbb{Z}-(3\mathbb{Z})\circ (6\mathbb{Z})=6\mathbb{Z}-(36\mathbb{Z} \cup 54\mathbb{Z})$ 	$2\circ 3$=\{12,18\}$\subseteq 6\mathbb{Z}-(36\mathbb{Z} \cup 54\mathbb{Z})$, but neither 2 $\in 6\mathbb{Z}$ nor $3 \in 6\mathbb{Z}$. 
		\end{exmp}
	The following lemma is a generalization of Lemma 2.1 in \cite{Akray}.
\begin{lemma}
 Let $P$ be a proper hyperideal of a hyperring $(R,+, \circ)$. Then $P$ is an $I$-prime hyperideal if and only if $P / I P$ is weakly prime hyperideal in $R / I P$.
\end{lemma}
\begin{proof}
	 $(\Rightarrow)$ Let $P$ be an $I$-prime hyperideal in $(R,+, \circ)$. Let $a, b \subseteq R$ with $\{0\} \neq(a+$ $I P)(b+I P)=a \circ b+I P \in P / I P$. Then $a \circ b \subseteq P-I P$ implies $a \subseteq P$ or $b \subseteq P$, hence $a+I P \subseteq P / I P$ or $b+I P \subseteq P / I P$. So $P / I P$ is weakly prime hyperideal in $R / I P$.
	$(\Leftarrow)$ Suppose that $P / I P$ is weakly prime hyperideal in $R / I P$ and take $r, s \subseteq R$ such that $r \circ s \subseteq P-I P$. Then $\{0\} \neq r \circ s+I P=(r+I P)(s+I P) \subseteq P / I P$ so $r+I P \subseteq P / I P$ or $s+I P \subseteq P / I P$. Therefore $r \subseteq P$ or $s \subseteq P$. Thus $P$ is an $I$-prime hyperideal in $R$.
	\end{proof}
	Let $(R,+, \circ)$ be a hyperring and $x$ be an indeterminate. Then $(R[x],+,\bullet)$ is a polynomial multiplicative hyperring where $ax^n \bullet bx^m = (a \circ b) x^{n+m}$  (see \cite{ciampi}).
\begin{theorem}
If $P$ is an $I-$prime hyperideal of 	 $(R,+, \circ)$, then $P[x]$ is $I[x]-$prime hyperideal of $(R[x],+,\bullet)$ .		 
\end{theorem}
\begin{proof}
	Let $a(x) \bullet b(x) \subseteq P[x]-I[x]\bullet P[x] = P[x]-(IP)[x]  $. Without loss of generality, let $a(x)=cx^n$ and $b(x)=dx^m$, for $c,d \in R$. Thus $c \circ d x^{n+m} \subseteq P[x]$ so $c \circ d \subseteq P$ and $ c \circ dx^{n+m} \nsubseteq IP[x] $
implies $c \circ d \nsubseteq IP$. P $I-$prime hyperideal gives us $c \in P$ or $d \in P$. Hence $a(x)=cx^n \in P[x]$ or $b(x)=dx^m \in P[x]$ and so $P[x]$ is an $I[x]-$prime.
	\end{proof}
\begin{corollary}
	Let $P$ be an $I-$prime hyperideal of $R$. Then $P[x]$ is an $I-$prime hyperideal of $R[x]$.
\end{corollary}
\begin{theorem}
		Let $R$ be a hyperring and $ f:R\longrightarrow R $ be a good epiomorphism and let $P$ be an I-prime hyperideal of $R$ with $Ker f \subseteq P$. Then $f(P) $ is an $I-$prime hyperideal.
	\end{theorem}
\begin{proof}
	Firstly, we have to show that $f(P)$ is hyperideal of $R$.
	Let $\bar{r} \in R$ and $y \in f(P)$. Then $x=f^{-1}(y) \in P$ and there exists $r \in R$ such that $f(r)=\bar{r} $.
So  $\bar{r}.y=f(r).f(x)=f(r.x) \subseteq f(P)$.	
	Now let us show that $f(P)$ is an $I-$prime hyperideal. To do this, we have for all $x,y \in R$ there exist $a,b \in R$ such that $x=f(a)$, $y=f(b)$. Then $x.y=f(a).f(b)=f(a.b) \subseteq f(P)$, so $a.b \subseteq P+Ker f$. As $P$ is an $I-$prime hyperideal, $a \in P$ or $b \in P$, that is $x=f(a) \in f(P)$ or $y=f(b) \in f(P)$. So $f(P)$ is an $I-$prime hyperideal of $R$.
	\end{proof}
\begin{theorem}
	Let $(R,+, \circ)$ be a hyperring and $ f:R\longrightarrow R $ be a good homomorphism and let $Q$ be an I-prime hyperideal of $R$. Then $f^{-1}(Q) $ is an $I-$prime hyperideal.
\end{theorem}
\begin{proof}
	Let $ a \circ b  \subseteq f^{-1}(Q) $. Then
$	f(a\circ b) = f(a) \circ f(b) \subseteq Q $ because $f$ is a good homomorphism. As $Q$ is $I-$prime hyperideal, $ f(a) \in Q$  or $ f(b) \in Q   $. So, $ a \in f^{-1}(Q)$  or $ b \in f^{-1}(Q)$ and hence
 $f^{-1}(Q)$ is an $I-$prime hyperideal of $R$.
	\end{proof}
The following theorem generalizes Theorem 2.2 of \cite{Akray}.
	\begin{theorem} \label{qqqq}
		(1) Let $I \subseteq J$ be two hyperideals of a multiplicative hyperring $R$. If $P$ is an $I-$prime hyperideal of $R$, then it is a $J-$prime hyperideal.

		(2) Let $R$ be a commutative multiplicative hyperring and $P$ an $I$-prime hyperideal that is not prime hyperideal, then $P^{2} \subseteq I P$. Thus, an I-prime hyperideal $P$ with $P^{2} \nsubseteq I P$ is a prime hyperideal.
	\end{theorem}
\begin{proof}
	(1) The proof comes from the fact that if $I \subseteq J$, then $P-J P \subseteq$ $P-I P$.

	(2) Suppose that $P^{2} \nsubseteq I P$, we show that $P$ is prime hyperideal. Let $a b \subseteq P$ for $a, b \in R$. If $a b \nsubseteq I P$, then $P$  $I$-prime gives $a \in P$ or $b \in P$. So assume that $a b \subseteq I P$. First, suppose that $a P \nsubseteq I P$; say $a x \nsubseteq I P$ for some $x \in P$. Then $a(x+b) \subseteq P-I P$. So $a \in P$ or $x+b \in P$ and hence $a \in P$ or $b \in P$. Hence we can assume that $a P \subseteq I P$ and in a similar way we can assume that $b P \subseteq I P$. Since $P^{2} \nsubseteq I P$, there exist $y, z \in P$ with $y z \nsubseteq  I P$. Then $(a+y)(b+z) \subseteq P-I P$. So $P$ $I$-prime gives $a+y \in P$ or $b+z \in P$ and hence $a \in P$ or $b \in P$. Therefore $P$ is a prime hyperideal of $R$ see also \cite{Akray} .
	\end{proof}
\begin{corollary}
 Let $P$ be an I-prime hyperideal of a hyperring $R$ with $I P \subseteq P^{3}$. Then $P$ is $\cap_{i=1}^{\infty} P^{i}$-prime hyperideal.
\end{corollary}
\begin{proof}
	 If $P$ is prime hyperideal, then $P$ is $\cap_{i=1}^{\infty} P^{i}$-prime hyperideal. Assume that $P$ is not prime hyperideal. By Theorem $2.5,  P^{2} \subseteq I P \subseteq P^{3}$. Thus $I P=P^{n}$ for each $n \geq 2$. So $\cap_{i=1}^{\infty} P^{i}=P \cap P^{2}=P^{2}$ and $\left(\cap_{i=1}^{\infty} P^{i}\right) P=P^{2} P=P^{3}=I P$. Being $P$ is $I$-prime hyperideal implies $P$ is $\cap_{i=1}^{\infty} P^{i}$-prime hyperideal.
		\end{proof}
\begin{remark}
 Let $P$ be an $I$-prime hyperideal. Then $P \subseteq \sqrt{I P}$ or $\sqrt{I P} \subseteq P$. If $P \varsubsetneqq \sqrt{I P}$, then $P$ is not prime hyperideal since otherwise $I P \subseteq P$ implies $\sqrt{I P} \subseteq \sqrt{P}=P$. While if $\sqrt{I P} \subsetneq P$, then $P$ is a prime hyperideal.
		Now we give a way to construct $I$-prime ideals $P$ when $\cap_{i=1}^{\infty} P^{i} \subseteq$ $I P \subseteq P^{3}$.

	\end{remark}
\begin{corollary}
	 Let $P$ be an $I-$prime hyperideal of a hyperring $R$ which is not prime hyperideal. Then $\sqrt{P}=\sqrt{I P}$.
\end{corollary}
\begin{proof}
	By Theorem 2.5, $P^{2} \subseteq I P$ and hence $\sqrt{P}=\sqrt{P^{2}} \subseteq \sqrt{I P}$. The other containment always holds.
	\end{proof}
\begin{remark}
	Assume that $P$ is an $I-$prime hyperideal, but not prime. Then by Theorem 2.5, if $I P \subseteq P^{2}$, then $P^{2}=I P$. In particular, if $P$ is weakly prime hyperideal (0-prime) but not prime hyperideal, then $P^{2}=\{0\}$. Suppose that $I P \subseteq P^{3}$. Then $P^{2} \subseteq I P \subseteq P^{3}$; So $P^{2}=P^{3}$ and thus $P^{2}$ is an idempotent.
	\end{remark}

\begin{lemma}
	 If $P$ is an $I-$primary hyperideal of a hyperring $R$, then $\sqrt{P}$ is a $\sqrt{I}-$prime hyperideal of $R$.
\end{lemma}
\begin{proof}
	Let $a b \subseteq \sqrt{P}-\sqrt{I}\sqrt{P} =\sqrt{P}-\sqrt{IP}$ for $a ,b \in R$. Then $(a b)^n=a^n b^n \subseteq P$ for some $n \in \mathbb{N}$ and $(a b)^m \nsubseteq IP$ for all $m \in \mathbb{N}$. So $a^n b^n \subseteq P-IP$ and as $P$ is an $I-$primary hyperideal of $R$, $a^n \subseteq P$ or $b^n \subseteq \sqrt{P}$, that is $a \in \sqrt{P}$ or $b \in \sqrt{P}$ which means that $\sqrt{P}$ is a $\sqrt{I}-$prime hyperideal of $R$.
\end{proof}

The following theorem generalizes the result  \cite[Theorem 2.8]{Akray}.
\begin{theorem}
	 (1) Let $R$ and $S$ be two commutative multiplicative hyperrings and $P$ be $\{0\}-$prime hyperideal of $R$. Then $P \times S$ is I-prime hyperideal of $R \times S$ for each hyperideal $I$ of $R \times S$ with $\cap_{i=1}^{\infty}(P \times S)^{i} \subseteq I(P \times S) \subseteq P \times S$.\\
	(2) Let $P$ be a finitely generated proper hyperideal of a commutative hyperring $R$. Assume $P$ is an $I$-prime hyperideal with $I P \subseteq P^{3}$. Then either $P$ is $\{0\} -$prime or $P^{2} \neq \{0\}$ is idempotent and $R$ decomposes as $T \times S$ where $S=P^{2}$ and $P=J \times S$ where $J$ is a $\{0\}-$prime. Thus $P$ is I-prime hyperideal for $\cap_{i=1}^{\infty} P^{i} \subseteq I P \subseteq$ $P$.
\end{theorem}
\begin{proof}
	 (1) Let $R$ and $S$ be two commutative hyperrings and $P$ be a $\{0\}-$prime hyperideal of $R$. Then $P \times S$ need not be a $\{0\}-$prime hyperideal of $R \times S$; In fact, $P \times S$ is $\{0\}-$prime if and only if $P \times S$ (or equivalently $P)$ is prime hyperideal. However, $P \times S$ is an $I$-prime hyperideal for each $I$ with $\cap_{i=1}^{\infty}(P \times S)^{i} \subseteq I(P \times S)$. If $P$ is prime hyperideal, then $P \times S$ is a prime hyperideal and thus is $I$-prime for all $I$. Assume that $P$ is not a prime hyperideal. Then $P^{2}=\{0\}$ and $(P \times S)^{2}=\{0\} \times S$. Hence $\cap_{i=1}^{\infty}(P \times S)^{i}=\cap_{i=1}^{\infty} P^{i} \times S=\{0\} \times S$. Thus $P \times S-\cap_{i=1}^{\infty}(P \times S)^{i}=P \times S-\{0\} \times S=(P-\{0\}) \times S$. Since $P$ is $\{0\}-$prime hyperideal, $P \times S$ is $\cap_{i=1}^{\infty}(P \times S)^{i}$-prime hyperideal and as $\cap_{i=1}^{\infty}(P \times S)^{i} \subseteq I(P \times S)$, $P \times S$ is $I$-prime hyperideal.
	\\

		(2) If $P$ is a prime hyperideal, then $P$ is $\{0\}-$prime. So we can assume that $P$ is not prime hyperideal. Then $P^{2} \subseteq I P$ and hence $P^{2} \subseteq I P \subseteq P^{3}$. So $P^{2}=P^{3}$. Hence $P^{2}$ is idempotent. Since $P^{2}$ is finitely generated, $P^{2}=<e>$ for some idempotent $e \in R$. Suppose $P^{2}=\{0\}$. Then $I P \subseteq P^{3}=\{0\}$. So $I P=\{0\}$ and hence $P$ is $\{0\}-$prime. So assume $P^{2} \neq \{0\}$. Put $S=P^{2}=$ <e> and $T=<1-e>$, so $R$ decomposes as $T \times S$ where $S=P^{2}$. Let $J=P(1-e)$, so $P=J \times S$ where $J^{2}=(P(1-e))^{2}=P^{2}(1-e)^{2}=<$ $e><1-e>=\{0\}$. We show that $J$ is $\{0\}-$prime hyperideal. Let $a \circ b \subseteq J-\{0\}$, so $(a, 1)(b, 1)=(a \circ b, 1) \subseteq J \times S-(J \times S)^{2}=J \times S-\{0\} \times S \subseteq P-I P$. Since $I P \subseteq P^{3}$ implies $I P \subseteq P^{3}=(J \times S)^{3}=\{0\} \times S$. Hence $(a, 1) \in P$ or $(b, 1) \in P$ so $a \in J$ or $b \in J$. Therefore $J$ is a $\{0\}-$prime hyperideal.
\end{proof}
\begin{corollary}
 Let $(R,+, \circ)$ be an indecomposable commutative hyperring and $P$ a finitely generated $I$-prime hyperideal of $(R,+, \circ)$, where $I P \subseteq P^{3}$. Then $P$ is a $\{0\}-$prime hyperideal.
\end{corollary}
\begin{corollary}
 Let $(R,+, \circ)$ be a Noetherian integral hyperdomain. A proper hyperideal $P$ of $R$ is prime hyperideal if and only if $P$ is $P^{2}$-prime hyperideal.
\end{corollary}
The next theorem is a generalization of \cite[Theorem 2.12]{Akray}.

\begin{theorem} \label{qqq}
	 Let $P$ be a proper hyperideal of a hyperring $R$. Then the following assertions are equivalent:
	 \begin{enumerate}
	 	\item  $P$ is I-prime hyperideal.
	 	\item For $r \in R-P$, $(P: r)=P \cup(I P: r)$.
	 	\item  For $r \in R-P$, $(P: r)=P$ or $(P: r)=(I P: r)$.
	 	\item For hyperideals $J$ and $K$ of $R$, $J K \subseteq P$ and $J K \nsubseteq I P$ imply $J \subseteq P$ or $K \subseteq P$.
	 \end{enumerate}
	
\end{theorem}
\begin{proof}
	 (1) $\Rightarrow$ (2) Suppose $r \in R-P$. Let $s \in(P: r)$, so $r s \subseteq  P$. If $r s \subseteq  P-I P$, then $s \in P$. If $r s \subseteq  I P$, then $s \in(I P: r)$. So $(P: r) \subseteq P \cup(I P: r)$. The other containment always holds.

	(2) $\Rightarrow$ (3) Note that if a hyperideal is a union of two hyperideals, then it is equal to one of them.

	$(3) \Rightarrow(4)$ Let $J$ and $K$ be two hyperideals of $R$ with $J K \subseteq P$. Assume that $J \nsubseteq P$ and $K \nsubseteq P$. We claim that $J K \subseteq I P$. Suppose $r \in J$. First, let $r \notin P$. Then $r K \subseteq P$ gives $K \subseteq(P: r)$. Now $K \nsubseteq P$, so $(P: r)=(I P: r)$. Thus $r K \subseteq I P$. Next, let $r \in J \cap P$. Choose $s \in J-P$. Then $r+s \in J-P$. By the first case $s K \subseteq I P$ and so $(r+s) K \subseteq I P .$ Pick $t \in K .$ Then $r t=(r+s) t-s t  \subseteq  I P $ and $r K \subseteq I P .$ Hence $J K \subseteq I P$.

	$(4) \Rightarrow(1)$ Let $r s \in P-I P .$ Then $(r)(s) \subseteq P$. But $(r)(s) \nsubseteq I P . $
So	$(r) \subseteq P$ or $(s) \subseteq P$ which means $r \in P$ or $s \in P .$
	\end{proof}

\begin{prop}
 Let $P$ be an $I$-prime hyperideal of a hyperring $R$ and  
   $J \subseteq P$ be a hyperideal of $R$. Then $P / J$ is $I-$prime hyperideal of $R / J$.
\end{prop}
\begin{proof}
	 Let $x, y \in R$ with $\bar{x} \circ \bar{y} \subseteq P / J-I(P / J)=P / J-(I P+J) / J$ where $\bar{x}$, $\bar{y}$ are the images of $x, y$ in $R / J$. Thus $x \circ y \subseteq P-I P$. So $x \in P$ or $y \in P$. Therefore $\bar{x} \in P / J$ or $\bar{y} \in P / J$. So $P / J$ is $I$-prime hyperideal.
	\end{proof}
Let $R_{1}$ and $R_{2}$ be two hyperrings. It is known that the prime hyperideals of $R_{1} \times R_{2}$ have the form $P \times R_{2}$ or $R_{1} \times Q$, where $P$ is a prime hyperideal of $R_{1}$ and $Q$ is a prime hyperideal of $R_{2}$. We next generalize this result to $I$-prime hyperideals.

\begin{theorem} \label{qqqqq}
Let $R_{i}$ be a hyperring and $I_{i}$ a hyperideal of $R_{i}$ for $i=1,2$. Let $I=I_{1} \times I_{2}$. Then the I-prime hyperideals of $R_{1} \times R_{2}$ have exactly one of the following three types:\\
(1) $P_{1} \times P_{2}$, where $P_{i}$ is a proper hyperideal of $R_{i}$ with $I_{i} P_{i}=P_{i}$.\\
(2) $P_{1} \times R_{2}$ where $P_{1}$ is an $I_{1}$-prime hyperideal of $R_{1}$ and $I_{2} R_{2}=R_{2}$.\\
(3) $R_{1} \times P_{2}$, where $P_{2}$ is an $I_{2}$-prime hyperideal of $R_{2}$ and $I_{1} R_{1}=R_{1}$.	
	\end{theorem}
\begin{proof}
	We first prove that a hyperideal of $R_{1} \times R_{2}$ having one of these three types is $I$-prime hyperideal. The first type is clear since $P_{1} \times P_{2}-I\left(P_{1} \times P_{2}\right)=$ $P_{1} \times P_{2}-\left(I_{1} P_{1} \times I_{2} P_{2}\right)=\phi$. Suppose that $P_{1}$ is $I_{1}$-prime hyperideal and $I_{2} R_{2}=R_{2}$. Let $(a, b)(x, y) \subseteq P_{1} \times R_{2}-\left(I_{1} P_{1} \times I_{2} R_{2}\right)=P_{1} \times R_{2}-\left(I_{1} P_{1} \times R_{2}\right)=$ $\left(P_{1}-I_{1} P_{1}\right) \times R_{2}$. Then $a x \subseteq P_{1}-I_{1} P_{1}$ implies that $a \in P_{1}$ or $x \in P_{1}$, so $(a, b) \in P_{1} \times R_{2}$ or $(x, y) \in P_{1} \times R_{2}$. Hence $P_{1} \times R_{2}$ is $I$-prime hyperideal. Similarly we can prove the last case.
	Next, let $P_{1} \times P_{2}$ be $I$-prime and $a b  \subseteq P_{1}-I_{1} P_{1}$. Then $(a, 0)(b, 0)=$ $(a b, 0) \in P_{1} \times P_{2}-I\left(P_{1} \times P_{2}\right)$, so $(a, 0) \in P_{1} \times P_{2}$ or $(b, 0) \in P_{1} \times P_{2}$, that is, $a \in P_{1}$ or $b \in P_{1}$. Hence $P_{1}$ is $I_{1}$-prime. Likewise, $P_{2}$ is $I_{2}$-prime.

	Assume that $P_{1} \times P_{2} \neq I_{1} P_{1} \times I_{2} P_{2}$, say $P_{1} \neq I_{1} P_{1}$. Let $x \in P_{1}-I_{1} P_{1}$ and $y \in P_{2}$. Then $(x, 1)(1, y)=(x, y) \in P_{1} \times P_{2}$. So $(x, 1) \in P_{1} \times P_{2}$ or $(1, y) \in P_{1} \times P_{2}$. Thus $P_{2}=R_{2}$ or $P_{1}=R_{1}$. Assume that $P_{2}=R_{2}$. Then $P_{1} \times R_{2}$ is $I$-prime, where $P_{1}$ is $I_{1}$-prime.
	
	\end{proof}

\section{$n-$absorbing $I-$prime hyperideals}
We start this section by the definition of $n-$absorbing $I-$prime hyperideals. 

\begin{definition}
	 A proper hyperideal $P$ of a hyperring $R$ is a $2-$absorbing $I-$prime hyperideal if for $x_1,x_2,x_3 \in R$ such that $x_1 x_2 x_3 \subseteq P-IP$, then $x_1 x_2 \subseteq P$ or $x_1 x_3 \subseteq P$  or $x_2 x_3 \subseteq P$. A proper hyperideal $P$ of $R$ is an $n-$absorbing $I-$prime hyperideal if for $x_1, \cdots,x_{n+1} \in R$ such that $x_1 \cdots x_{n+1} \subseteq P-IP$, then $x_1 \cdots x_{i-1} x_{i+1} \cdots  x_{n+1} \subseteq P$ for some $i \in \{1,2, \cdots ,n+1\}$.
\end{definition}
It is clear that the class of $n-$absorbing $I-$prime hyperideals contains properly the class of $n-$absorbing hyperideals. As we can see this in the following example.
\begin{exmp}
 Let $K$ be a hyperfield and $R=K[x_1,\cdots,x_{n+2}]$ be a polynomial multiplicative hyperring. Consider the hyperideals $P=<x_1 \cdots x_{n+1} , x_1^2 \cdots x_n , x_1^2 x_{n+2}>$ and $I=<x_1 \cdots x_n>$. So $P-IP=<x_1 \cdots x_{n+1}, x_1^2 \cdots x_n ,  x_1^2 x_{n+2} >$ $-$ $< x_1 \cdots x_{n+1},  x_1^2 \cdots x_n ,  x_1^2 \cdots x_n x_{n+2}>$. Hence $P$ is an $n-$absorbing $I-$prime hyperideal but not $n-$absorbing hyperideal.

\end{exmp}
\begin{lemma}
	Let $P$ be an $I-$prime hyperideal of $R$ and $K$ be a subset of $R$. For any $a \in R, a  K \subseteq P,\,aK\nsubseteq IP$ and $a \notin P$ implies that $K \subseteq P$. (or $a K \subseteq P$ and $K \nsubseteq P$ imply that $a \in P$ ).
\end{lemma}
\begin{proof}
	Let $a K \subseteq P$ and $a \notin P$ for any $a \in R$. Then we have $a K=\cup a k_{i} \subseteq P$ for all $k_{i} \in K$. Hence $a k_{i} \subseteq P$ and $ak_i\nsubseteq IP$ for all $k_{i} \in K$. Since $P$ is an $I-$prime hyperideal and $a \notin P$, $k_{i} \in P,\,\, \forall k_{i} \in K$. Thus $K \subseteq P$.
\end{proof}
\begin{lemma}
	Let $P$ be an $I-$prime hyperideal of $R$ and $A, B$ be subsets of $R$. If $A B \subseteq P$ and $AB\nsubseteq IP$, then $A \subseteq P$ or $B \subseteq P$.
\end{lemma}
\begin{proof}
	Assume that $A B \subseteq P, AB\nsubseteq IP$ and $   A \nsubseteq P$, $B \nsubseteq P$. Since $A B=\bigcup a_{i} b_{i} \subseteq P$, $a_{i} b_{i} \subseteq P$, for $a_{i} \in A, b_{i} \in B$. And as $A \nsubseteq P$ and $B \nsubseteq P$, we have $x \notin P$ and $y \notin P$ for some $x \in A, y \in B$. Then $x y \subseteq A B \subseteq P$ and $xy \nsubseteq IP$. From being   $P$ an $I-$prime hyperideal, we have $x \in P$ or $y \in P$  which is a contradiction. Thus $A \subseteq P$ or $B \subseteq P$.	
\end{proof}
\vspace{0.5cm}
Every $I-$prime hyperideal is a $2-$absorbing $I-$prime hyperideal. Since for  $(ab)c \subseteq P-IP$, we have $a b\subseteq P$ or $b c \subseteq P$. If $a b \nsubseteq P$ then by $I-$prime hyperideal of $P$, we have $c \in P$ and so $a c \in P$ or $b c \in P$. Hence $P$ is a $2-$absorbing $I-$prime hyperideal of $R$.

\begin{lemma}
	Let $P$ be a hyperideal of $R$ and $P_{1}, P_{2}, \ldots, P_{n}$ be $2-$absorbing primary hyperideals of $R$ such that $\sqrt{P_{i}}=P$ for all $\mathrm{i}=1, \ldots, \mathrm{n}$. Then $\bigcap_{i=1}^{n} P_{i}$ is a 2-absorbing $I-$prime hyperideal and $\bigcap_{i=1}^{n} P_{i}=P$.
\end{lemma}
\begin{proof}
	Assume $P=\bigcap_{i=1}^{n} P_{i}$ and so $\sqrt{P}=\sqrt{\cap_{i=1}^n P_i}=\cap_{i=1}^n \sqrt{P_i}=P$. Let  $x y z \subseteq P-IP$ with $x y \nsubseteq P$, for $x, y, z \in R$. Thus $x y \nsubseteq P_{i}$ for some $i=1,2,\cdots, n$. From being $P_{i}$ a 2-absorbing primary hyperideal and $x y z \subseteq P-IP \subseteq P_{i}$, hence $x z \subseteq \sqrt{P_{i}}=P$ or $y z \subseteq \sqrt{P_{i}}=P$ which means that $P$ is a 2-absorbing $I-$prime hyperideal of $R$. 
\end{proof}
\begin{theorem}
		Let $h: R \rightarrow L$ be a bijective good homomorphism of hyperrings and $P$ be a 2-absorbing $I-$prime hyperideal of $L$. Then $h^{-1}(P)$ is a 2-absorbing $h^{-1}(I)-$prime hyperideal of $R$.
\end{theorem}
\begin{proof}
	Suppose that $a b c \subseteq h^{-1}(P), h^{-1}(I)h^{-1}(P)=h^{-1}(P)- h^{-1}(IP)$, for $a, b, c \in R$. So $h(a b c)=h(a) h(b) h(a) \subseteq P$ and $h(a b c) \nsubseteq IP$. From being $P$  a 2-absorbing $I-$prime hyperideal, we have $h(a) h(b) \subseteq P$ or $h(a) h(c) \subseteq P$ or $h(b) h(c) \subseteq P$, that is $h(a b) \subseteq P$ or $h(a c) \subseteq P$ or $h(b c) \subseteq P$ which implies $a b \subseteq$ $h^{-1}(P)$ or $a c \subseteq h^{-1}(P)$ or $b c \subseteq h^{-1}(P)$. So $h^{-1}(P)$ is a 2-absorbing $h^{-1}(I)-$prime hyperideal of $R$.	
\end{proof}
\begin{theorem}
	Suppose that $P$ is an $n-$absorbing $I-$prime hyperideal of $R$. Then $\sqrt{P}$ is an $n-$absorbing $\sqrt{I}-$prime hyperideal of $R$ and $a^{n} \subseteq P$ for all $a \in \sqrt{P}$.	
\end{theorem}
\begin{proof}
	Let $a \in \sqrt{P}$. Then $a^{m} \subseteq P$ for some $m \in \mathbb{N}$. If $m \leq n$, we are done. If $m>n$, by using the $n$-absorbing $I-$prime property on products $ a^{m}$, we conclude that $a^{n} \subseteq P$. Now, consider $a_{1}  \cdots  a_{n+1} \subseteq \sqrt{P} - \sqrt{I}\sqrt{P}=\sqrt{P}-\sqrt{IP}$ for $a_{1}, \cdots, a_{n+1} \in R$. Thus $\left(a_{1}  \ldots  a_{n+1}\right)^{n}=a_{1}^{n}  \cdots  a_{n+1}^{n} \subseteq P$. If $a_{1}^{n}  \cdots  a_{n+1}^{n} \subseteq IP$, then $a_1 \cdots  a_{n+1} \subseteq  \sqrt{IP}$ which is a contradiction. Hence $a_{1}^{n}  \cdots  a_{n+1}^{n} \subseteq P - IP$ and $P$  $n$-absorbing $I-$prime hyperideal gives us the desired.
\end{proof}
\label{We use this corollary afetr theorm 2.2.}

\begin{lemma}
	
Let $P_i$ be an $n_i-$absorbing $I-$prime hyperideal of a hyperring  $R$ for $i=1,2, \cdots ,m$ and $IP_i=IP_j$, for $ i\neq j$, Then $\cap_{i=1}^m P_i$ is an $n-$absorbing $I-$prime hyperideal where $n=\sum_{i=1}^{m}n_i$.
\end{lemma}
\begin{proof}
		Let $k>n$ and $x_1 \cdots x_k\subseteq \cap_{i=1}^m P_i-I\cap_{i=1}^m P_i$. Then by hypothesis for each $ i=1 \cdots m$, there exists a product of $n_i$ of these $k-$elements in $P_i$. Let $A_i$ be the collection of these elements and let $A=\cup_{i=1}^k A_i$. Thus $A$ has at most $n-$elements. Now, as $P_i$ is an $n-$absorbing $I-$prime hyperideal, the product of all elements of $A$ must be in each $P_i$ so $\cap P_i$ contains a product of at most $n-$elements and therefore it is an $n-$absorbing $I-$prime hyperideal of $R$.
\end{proof}

\begin{theorem} \label{3.7}
		Let $R=\prod_{i=1}^{n+1}R_i$ and $P$ be a proper non-zero hyperideal of $R$. If P is an $(n+1)-$absorbing $I-$prime hyperideal of $R$, then $P=P_1\times P_2\times \cdots \times P_{n+1}$ for some proper $n-$absobing $I_i-$prime hyperideals $P_1,\cdots ,P_{n+1}$ of $R_1,\cdots,R_{n+1}$ respectively, where  $I=\prod_{i=1}^{n+1}I_i$ and $I_i = R_i$, $ \forall i=1,2,\cdots,n+1.$	
	\end{theorem}

\begin{proof}
	Let $x_1,\cdots,x_{n+1} \in R$ with $x_1 \cdots x_{n+1} \subseteq P_1-I_1P_1$ and suppose by contrary that $P_1$ is not an $n-$absorbing $I_1-$prime hyperideal of $R_1$. Set $a_i=(x_i,1,1,\cdots,1)$ for $i=1,2,\cdots,n+1$ and $a_{n+2}=(1,0,\cdots,0)$. Then we have $a_1a_2 \cdots a_{n+2}=(x_1x_2 \cdots x_{n+1},0,0,\cdots ,0) \subseteq P-IP$
and $a_1 \cdots a_{i-1}a_{i+1} \cdots  a_{n+2}=(x_1x_2 \cdots x_{i-1}x_{i+1} \cdots x_{n+1},0, \cdots,0) \nsubseteq P$ for $i=1,\cdots ,n+1$, which contradicts with being $P$ an $(n+1)-$absorbing $I-$prime hyperideal. Hence $P_1$ must be an $n-$absorbing $I_1-$prime hyperideal of $R_1$. By similar arguments, we can show that $P_i$ is an $n-$absorbing $I_i-$prime hyperideal of  $R_i$ for $i=1,\cdots ,n+1$.
	\end{proof}
\begin{theorem}\label{qq}
	Let $R=\prod_{i=1}^{n+1} R_i$, where $R_i$	is a hyperring for $i \in \{1,\cdots,n+1\}$. If $P$ is an $n-$absorbing $I-$prime hyperideal of $R$, then either $P=IP$ or $P=P_1\times P_2 \times \cdots \times  P_{i-1} \times R_i \times P_{i+1} \cdots \times  P_{n+1}$ for some $i \in \{1,\cdots,n+1\}$ and if $P_j \ne R_i$ for $j \ne i$, then $P_j$ is an $n-$absorbing hyperideal in $R_i$.
\end{theorem}

\begin{proof}
	Let $P=\prod_{i=1}^{n+1} P_i$ be an $n-$absorbing $I-$prime hyperideal of $R$. Then there exists $(x_1,\cdots,x_{n+1}) \subseteq P-IP$, and so $(x_1,1,\cdots,1) (1,x_2,1$
	 \\$\cdots,1) \cdots (1,1,\cdots,1,x_{n+1})= (x_1,x_2,\cdots,x_{n+1}) \subseteq P-IP$. As $P$ is an $n-$absorbing $I-$prime hyperideal, we have $(x_1,x_2,\cdots ,x_{i-1},1,x_{i+1},\cdots $\\ $,x_{n+1}) \subseteq  P$ for some $i \in \{1,2,\cdots ,n+1\}$. Thus $(0,0,\cdots,0,1,0,\cdots,0)$ 
	\\$\in P$ and hence $P=P_1\times P_2 \times \cdots \times P_{i-1} \times R_i \times P_{i+1}  \cdots  \times P_{n+1}$. If $P_j \ne R_i$ for $j \ne i$, then we have to prove $P_j$ is an $n-$absorbing hyperideal of $R_i$. Let $i < j$ and take $x_1x_2\cdots x_{n+1} \subseteq P_j$. Then $(0,0,\cdots,0,1,0,\cdots,0,x_{1}x_{2}\cdots x_{n+1},0\cdots,0)$ $= (0,0,\cdots,1,0,\cdots,0,x_{1}\\
	,0\cdots,0) (0,0,\cdots,1,0,\cdots,0,x_{2},0\cdots,0) \cdots$ $(0,0,\cdots,1,0,\cdots,0,x_{n+1}\\
	,0\cdots,0) \subseteq  P-IP$. Since $P$ is an $n-$absorbing $I-$prime hyperideal, $(0,0,\cdots,0,1,0,\cdots,0,x_1x_2 \cdots x_{k-1} x_{k+1}\cdots x_{n+1}, 0,\cdots ,0) \in P$ for some $k \in \{1,2,\cdots,n+1\}$. Thus $x_1x_2 \cdots x_{k-1} x_{k+1}\cdots x_{n+1} \in P_j$ and hence $P_j$ is an $n-$absorbing hyperideal of $R_i$. We can do similar arguments for the case $i>j$.
\end{proof}

In the following result, we characterize hyperrings in which every proper hyperideal of $R$ is an $n-$absorbing $I-$prime hyperideal.	
\begin{theorem} \label{3.9}
	
	Let $\mid Max(R)\mid \geq n+1 \geq 2.$ Then each proper hyperideal of $R$ is an $n-$absorbing $I-$prime hyperideal if and only if each quotient of $R$ is a product of $(n+1)-$hyperfields.

\end{theorem}
\begin{proof}
	$(\Rightarrow)$ 
	Let $P$ be a proper hyperideal of $R$. Then $\frac{R}{IP}\cong F_1\times\cdots\times F_{n+1}$ and $\frac{P}{IP}\cong P_1\times\cdots\times P_{n+1}$, where $P_i$ is a hyperideal of $F_i$, $i=1,\cdots ,{n+1}$. If $P=IP$ then there is nothing to prove, otherwise we have $P_j=0,$ for at least one $j\in \{1,\cdots , {n+1}\}$ since $\frac{P}{IP}$ is proper. So, $\frac{P}{IP}$ is an $n-$absorbing $0-$prime hyperideal of $\frac{R}{IR}$ which means $P$ is an $n-$absorbing $I-$prime hyperideal of $R$.\\ 
	$(\Leftarrow)$ Let $m_1,\cdots , m_{n+1}$ be distinct maximal hyperideals of $R$. Then $m=\prod_{i=1}^{n+1}m_i$ is an $n-$absorbing $I-$hyperideal of $R$. we claim that $m$ is not an $n-$absorbing hyperideal. First, if $m_i\subseteq\cup_{i\neq j} m_j$, then there exist $m_j$ with $m_i\subseteq m_j$ by Prime Avoidance Lemma % \\
	% (Theorem A(Prime Avoidance Lemma). Suppose that $P_{1}, P_{2}, \ldots, P_{n}$ are ideals of a ring $R$ and $S$ is a subring $-1$ of $R$ with $S \subseteq \bigcup_{i=1}^{n} P_{i}$. If at most two of the $P_{i}$ 's are not prime and $S$ is either prime preserving or a right (left) ideal with respect to $\bigcup_{i=1}^{n} P_{i}$, then $S \subseteq \bigcup_{i=1}^{n} P_{i}$ is reducible. In particular, $S \subseteq P_{i}$ for some $P_{i}$.)()some i)\\
	  and this contradicts the maximality of $m_i$. Hence $m_i\nsubseteq\cup_{i\neq j}m_j$ and so, there exists $x_i\in m_i-\cup_{i\neq j}^{n+1}m_j$ so that $x_1 \cdots x_{n+1}\subseteq m$. If there exists $j\in \{1,\cdots ,n+1\}$ with $a=x_1x_2 \cdots x_{j-1}x_{j+1} \cdots x_{n+1}	\subseteq  m\subseteq m_j$, then $x_i\in m_j$ for some  $i\neq j$ which is a contradiction. Hence $m$ is not an $n-$absorbing hyperideal and so $m^{n+1}=Im$. Then by Chinese Remainder Theorem
	  % ( Chinese Remainder Theorem $(C R T)$
%	If the integers $n_{1}, n_{2}, \ldots, n_{k}$ are pair-wise relatively prime, then the system of simultaneous congruence
%	$$
%	\begin{gathered}
	%	x \equiv r_{1} \bmod n_{1} \\
	%	x \equiv r_{2} \bmod n_{2} \\
	%	\vdots \\
	%	x \equiv r_{k} \bmod n_{k}
%	\end{gathered}
%	$$
%	has a unique solution: $x=\sum_{i=1}^{k} r_{i} N_{i}^{-1} N_{i} \bmod N$ where;
%	$$
%	\begin{aligned}
%		&N=\prod_{i=1}^{k} n_{i} \\
%		&N_{i}=\frac{N}{n_{i}} \\
%		&N_{i}^{-1} N_{i} \equiv 1 \bmod n_{i}
%	\end{aligned} 	$$)\\
	 we have $\frac{R}{Im}\simeq\frac{R}{m_1^{n+1}}\times\frac{R}{m_2^{n+1}}\times...\times\frac{R}{m_{n+1}^{n+1}}$. Put $F_i= \frac{R}{m_1^{n+1}}$. If $F_i$ is not a hyperfield, then it has a nonzero proper hyperideal $H$ and so $0\times0\times \cdots \times0\times H\times0\times \cdots \times0$ is an $n-$absorbing $0-$hyperideal of $\frac{R}{Im}$. Thus, by Lemma \ref{qq} we have $H=F_i$ or $H=0$ which is impossible. Hence $F_i$ is a hyperfield.
	
\end{proof}

\begin{corollary}
	
	Suppose $\mid Max(R)\mid\geq n+1\geq2$.
	Then each proper hyperideal of $R$ is
	an $n-$absorbing  $0-$hyperideal if and only if $R\cong F_1\times \cdots \times F_{n+1}$, where $F_1, \cdots ,F_{n+1}$ are hyperfields.
	
\end{corollary}

\begin{theorem} \label{3.10}
	Let $P$ be an $n-$absorbing $I-$prime hyperideal of a hyperring $R$. Then there are at most $n$ prime hyperideals    of $R$ that are minimal over $P$.	
	
\end{theorem}
\begin{proof}
	
	Let $C=\{q_i: q_i \textnormal{ is a prime hyperideal minimal over} \,P \}$ and let $C$ has at least $n$ elements. Assume $q_1,\cdots, q_n \in C$ are distinct elements and $x_i \in q_i - \cup_{j\ne i} q_j $ for $i=1,\cdots,n$. By \cite[Theorem 2.1]{qwer}, there is a $ y_i \notin q_j$ such that $y_i x_i^{t_i} \subseteq P$ for $i=1,\cdots,n$ and for some positive integers $t_1,\cdots,t_n$. Since $x_i \notin \cap_{j=1}^n q_j$ and $P$ an $n-$absorbing $I-$prime hyperideal, we have $y_i x_i^{n-1} \in P$. As $x_i \notin \cap_{j=1}^n q_j$ and  $y_i x_i^{n-1}  \subseteq P \subseteq \cap_{j=1}^n q_j$, we get $y_i \in q_i-\cup_{j\ne i} q_j$, and so $y_i \notin \cap_{j=1}^n q_j$ for $i=1,\cdots,n$. Since $y_i x_i^{n-1}\subseteq P$, $\sum_{j=1}^{n} y_j \prod_{i=1}^n x_i^{n-1} \subseteq P$ and clearly $\sum_{j=1}^n y_j \notin q_i$, for $i=1,\cdots,n$, and being $P$ an $n-$absorbing $I-$prime hyperideal, we have $y_i x_i^{n-1} \in P$. As $x_i \notin \cap_{j=1}^n q_j$ and $y_i x_i^{n-1} \subseteq P \subseteq \cap_{i=1}^n q_i$, we get $y_i \in q_i-\cup_{j\neq i} q_j$ and so $y_i \notin \cap_{i=1}^n q_i $ for $i=1, 
	\cdots n$. Since $y_i x_i^{n-1} \subseteq P$, $\varSigma_{j=1}^n y_j \prod_{i=1}^nx_i^{n+1} \subseteq P$ and clearly $\varSigma_{j=1}^n y_j \notin q_i$, for $i=1, 
	\cdots n$ and being $P$ an $n-$absorbing $I-$prime hyperideal, we have $\prod_{i=1}^n x_i^{n-1} \subseteq P$. Now, suppose $q_{n+1} \in C$ such that $q_{n+1} \neq q_i$, for $i=1, \cdots ,n$ and consequently $z_i \in q_{n+1}$ for $i=1, \cdots ,n$ which is a contradiction. Therefore $C$ has at least $n$ elements.  
	
	\end{proof}

In a multiplicative hyperring $(R,+, \circ)$ a non empty subset $L$ of $R$ is called a multiplicative set whenever $a, b \in A \Rightarrow a \circ b \cap A \neq \phi$.

We can contract the localization of a multiplicative hyperring $R$ as follows:
Let $S$ be a multiplicative closed subset of $R$, that is, $S$ is closed under the hypermultiplication and contains the identity. Let $S^{-1}R$ be the set $(R \times S / \sim)$ of equivalence classes where 

$$(r_1, s_1) \sim (r_2,s_2) \iff \exists s \in S\,\, \text{such that } ss_1r_2=ss_2r_1.$$
Let $r/s$ be the equivalence class of $(r,s) \in R \times S$ under the equivalence relation $ \sim$. The operation addition and the   hyperoperation multiplication are defined by $$\frac{r_1}{s_1}+\frac{r_2}{s_2}= \frac{s_1r_2+s_2r_1}{s_1s_2} =\{ \frac{a+b}{c}: a \in s_1r_2, b \in s_2r_1, c \in s_1s_2 \} $$  $$ \frac{r_1}{s_1} \cdot {\frac{r_2}{s_2}}= \frac{r_1r_2}{s_1s_2}=  \{\frac{a}{b}, a\in r_1r_2, b \in s_1s_2\}.$$

Note that the localization map $f: R \to S^{-1}R$, $f(r)= \frac{r}{1}$ is a homomorphism of hyperrings. It is easy to see that the localization of a hyperideal is a hyperideal.
\begin{prop}
	Let $P$ be an $I-$prime hyperideal of $R$ with $S \cap P = \emptyset$. Then $S^{-1}P$ is an $S^{-1}I-$prime hyperideal of $S^{-1}R$.
\end{prop}
\begin{proof}
	$ \frac{r_1}{s_1} , \frac{r_2}{s_2}  \in S^{-1}R$ with  $\frac{r_1}{s_1} \frac{r_2}{s_2}= \frac{r_1r_2}{s_1s_2} \subseteq S^{-1}P-S^{-1}IS^{-1}P=S^{-1}P-S^{-1}(IP)$. For each $n \in r_1r_2, s \in s_1s_2$, we have $\frac{n}{s} \in \frac{r_1r_2}{s_1s_2}$ and $\frac{n}{s}=\frac{a}{t}$, where $a \in P$, $t \in S$. So there exists $q \in S$ such that $qtn=qsa$. Hence $qtn \subseteq P-IP$ and so $qr_1r_2 \subseteq P-IP$. As $P$ is an $I-$prime hyperideal, we have $qr_1 \subseteq P$ or $r_2 \in P$. Thus $\frac{r_1}{s_1}=\frac{qr_1}{qs_1} \in P$ or $\frac{r_2}{s_2} \in S^{-1}P$. Therefore $S^{-1}P$ is an $S^{-1}I-$prime hyperideal of $S^{-1}R$.
	\end{proof}
{\bf Questions} Readers can think about the follwoing subjects:
\begin{enumerate}
	\item $2-I-$primal hyperideals
	\item $2-I-$primal hypersubmodules
	\item $2-I-$prime hypersubmodules
	\end{enumerate}

{\bf Conclosion.} In this article we transfer notions $I-$prime ideals and $n-$absorbing $I-$ideals in multiplicative hyperrings and named them $I-$prime hyperideal and $n-$absorbing $I-$prime hyperideal. We study some properties of such two concepts and we see that they have analogous properties of prime ideals. During the study, we found out similar consepts that one can think about like $2-I-$primal hyperideals, $2-I-$primal hypersubmodules and $2-I-$prime hypersubmodules.

\vspace*{.1 in}
{\bf Acknowleghments.}
The authors would like to express their sincere thanks to the anonymous referees for their valuable comments and  suggestions on the earlier version of the manuscript which improved it a lot.

\end{document}